\documentclass[a4paper,12pt]{amsart}
\usepackage{geometry}
\usepackage[utf8]{inputenc}
\usepackage[british]{babel}

\usepackage{amsthm,amssymb,amsmath,graphicx,mathtools, amsaddr}
\usepackage[shortlabels, inline]{enumitem} 
\usepackage[abbrev]{amsrefs}
\usepackage{hyperref}

\allowdisplaybreaks
\newtheorem{theorem}{Theorem}
\theoremstyle{plain}
\newtheorem{claim}[theorem]{Claim}
\newtheorem{conjecture}[theorem]{Conjecture}
\newtheorem{corollary}[theorem]{Corollary}
\newtheorem{lemma}[theorem]{Lemma}
\numberwithin{equation}{section}
\numberwithin{theorem}{section}

\usepackage[usenames,dvipsnames]{color}

\def\schrom{\chi_{\textnormal{s}}}
\def\dfn{:=}
\def\eps{\varepsilon}
\DeclareMathOperator{\expectation}{\mathbf{E}}
\DeclareMathOperator{\prob}{Pr}

\newenvironment{proofclaim}[1][Proof of the claim]{\begin{proof}[#1]}{\end{proof}}

\setenumerate[1]{label=\textnormal{(\roman*)}}
\setenumerate[2]{label=\textnormal{(\alph*)}}

\title{An asymptotic bound for the strong chromatic number} 
\author{Allan Lo and Nicolás Sanhueza-Matamala}
\stepcounter{g@author}
\address[Allan Lo, Nicolás Sanhueza-Matamala]{School of Mathematics,\\ University of Birmingham,\\ Birmingham B15 2TT,\\ United Kingdom}
\email{s.a.lo@bham.ac.uk, NIS564@bham.ac.uk}
\thanks{The research leading to these results was partially supported by EPSRC, grant no. EP/P002420/1 (A.~Lo) and the Becas Chile scholarship scheme from CONICYT (N.~Sanhueza-Matamala).}
\begin{document}

\begin{abstract}
	The strong chromatic number~$\schrom(G)$ of a graph~$G$ on~$n$ vertices is the least number~$r$ with the following property: after adding~$r \lceil n/r \rceil - n$ isolated vertices to~$G$ and taking the union with any collection of spanning disjoint copies of~$K_r$ in the same vertex set, the resulting graph has a proper vertex-colouring with~$r$ colours.
	We show that for every~$c > 0$ and every graph~$G$ on~$n$ vertices with $\Delta(G) \ge cn$, $\schrom(G) \leq (2 + o(1)) \Delta(G)$, which is asymptotically best possible.
	
\end{abstract}

\maketitle
	
\section{Introduction}

Let~$r$ be a positive integer.
Let~$G$ be a graph on~$n$ vertices, where~$r$ divides~$n$.
We say that~$G$ is \emph{strongly $r$-colourable} if it can be properly $r$-coloured after taking the union of~$G$ with any collection of spanning disjoint copies of~$K_r$ in the same vertex set.
Equivalently,~$G$ is strongly $r$-colourable if for every partition~$\{ V_1, \dotsc, V_k \}$ of~$V(G)$ with classes of size~$r$, there is a proper vertex colouring of~$G$ using~$r$ colours with the additional property that every~$V_i$ receives all of the~$r$ colours.
If~$r$ does not divide~$n$, we say that~$G$ is \emph{strongly $r$-colourable} if the graph obtained by adding~$r \lceil n / r \rceil - n$ isolated vertices to~$G$ is $r$-strongly colourable.
The \emph{strong chromatic number~$\schrom(G)$} of~$G$ is the minimum~$r$ such that~$G$ is $r$-strongly colourable.
This notion was introduced independently by Alon~\cite{Alon1988} and Fellows~\cite{Fellows1990}.

One of the first problems related to the strong chromatic number was the cycles-plus-triangles problem of Erd\H{o}s (see~\cite{FleischnerStiebitz1997}), who asked (in an equivalent form) if~$\schrom(C_{3m}) \leq 3$, where~$C_{3m}$ is the cycle on~$3m$ vertices.
This was answered affirmatively by Fleischner and Stiebitz~\cite{FleischnerStiebitz1992} and independently by Sachs~\cite{Sachs1993}.

It is an open problem to find the best bound on~$\schrom(G)$ in terms of~$\Delta(G)$.
Alon~\cite{Alon1992} proved that~$\schrom(G) \leq c \Delta(G)$ for some constant~$c > 0$.
Haxell~\cite{Haxell2004} showed that~$c = 3$ suffices and later~\cite{Haxell2008} that~$c \leq 11/4 + \eps$ suffices given~$\Delta(G)$ is large enough with respect to~$\eps$.
On the other hand, there are examples showing~$c \ge 2$ is necessary (see, e.g.,~\cite{AxenovichMartin2006}).
It is conjectured (first explicitly stated by Aharoni, Berger and Ziv~\cite[Conjecture 5.4]{AharoniBergerZiv2007}) that this lower bound is also tight.

\begin{conjecture} \label{conjecture:strongchromaticnumber}
	For every graph~$G$, $\schrom(G) \leq 2 \Delta(G)$.
\end{conjecture}

Conjecture~\ref{conjecture:strongchromaticnumber} is known to be true for graphs~$G$ on~$n$ vertices with~$\Delta(G) \ge n/6$, proven by Axenovich and Martin~\cite{AxenovichMartin2006} and independently by Johansson, Johansson and Markström~\cite{JohanssonJohanssonMarkstrom2010}.

A fractional version of Conjecture~\ref{conjecture:strongchromaticnumber} was proven by Aharoni, Berger and Ziv~\cite{AharoniBergerZiv2007}.
We say that a graph on~$n$ vertices is \emph{fractionally strongly $r$-colourable} if after adding $r \lceil n/r \rceil - n$ isolated vertices and taking the union with any collection of spanning copies of~$K_r$ in the same vertex set, the graph is fractionally $r$-colourable.

\begin{theorem}[Aharoni, Berger and Ziv {\cite{AharoniBergerZiv2007}}] \label{theorem:fractionalcolouring}
	Every graph~$G$ is fractionally strongly $r$-colourable, for every~$r \ge 2 \Delta(G)$.
\end{theorem}

We prove that Conjecture~\ref{conjecture:strongchromaticnumber} is asymptotically true if~$\Delta(G)$ is linear in~$|V(G)|$.

\begin{theorem} \label{theorem:strongcolouring}
	For all~$c, \eps > 0$, there exists~$n_0 = n_0(c, \eps)$ such that the following holds: if~$G$ is a graph on~$n \ge n_0$ vertices with~$\Delta(G) \ge c n$, then~$\schrom(G) \leq (2 + \eps) \Delta(G)$.
\end{theorem}

Given a graph~$G$ and a partition~$\mathcal{P} = \{ V_1, \dotsc, V_k \}$ of~$V(G)$, we make the following definitions.
A subset~$S \subseteq V(G)$ is \emph{$\mathcal{P}$-legal} if~$|S \cap V_i| \leq 1$ for every~$i \in [k]$.
A \emph{transversal of~$\mathcal{P}$} is a $\mathcal{P}$-legal set of cardinality~$|\mathcal{P}|$.
An \emph{independent transversal of~$\mathcal{P}$} is a transversal of~$\mathcal{P}$ which is also an independent set in~$G$.
We will write \emph{transversal} and \emph{independent transversal} if~$G$ and~$\mathcal{P}$ are clear from the context.
To prove Theorem~\ref{theorem:strongcolouring} it suffices to show that given any partition~$\mathcal{P}$ of~$V(G)$ with classes of size~$r \ge (2 + \eps) \Delta(G)$,~$V(G)$ can be partitioned into independent transversals of~$\mathcal{P}$.
Moreover, since Conjecture~\ref{conjecture:strongchromaticnumber} is known to be true for graphs on~$n$ vertices with~$\Delta(G) \ge n/6$, we might restrict ourselves to study graphs with~$\Delta(G) \leq n/6$, and in such graphs any partition~$\mathcal{P}$ of~$V(G)$ with parts of size~$r = (2 + \eps) \Delta(G) < 3 \Delta(G)$ will have at least~$3$ classes.
Thus Theorem~\ref{theorem:strongcolouring} is implied by the following theorem.

\begin{theorem} \label{theorem:main}
	For all integers~$k \ge 3$ and~$\eps > 0$, there exists~$r_0 = r_0(k, \eps)$ such that the following holds for all~$r \ge r_0$:
	if~$G$ is a graph and~$\mathcal{P}$ is a partition of~$V(G)$ with~$k$ classes of size~$r \ge (2 + \eps) \Delta(G)$, then there exists a partition of~$V(G)$ into independent transversals of~$\mathcal{P}$.
\end{theorem}

By considering the complement graph, Theorem~\ref{theorem:main} easily yields the following corollary.
A \emph{perfect $K_k$-tiling} of a graph~$G$ is a spanning subgraph of~$G$ with components which are complete graphs on~$k$ vertices.

\begin{corollary}
	For all integers~$k \ge 3$ and~$\eps > 0$, there exists~$n_0 = n_0(k, \eps)$ such that the following holds:
	if~$n \ge n_0$ and~$G$ is a $k$-partite graph with classes of size~$n$ and~$\delta(G) \ge (k - 3/2 + \eps)n$, then~$G$ has a perfect $K_k$-tiling.
\end{corollary}

To prove Theorem~\ref{theorem:main} we use the absorption method, which was first introduced in a systematic way by R\"odl, Ruci\'nski and Szemer\'edi~\cite{RodlRucinskiSzemeredi2006} (although similar ideas were used previously, e.g. by Krivelevich \cite{Krivelevich1997}).
In Section~\ref{section:absorption} we find a small absorbing set, that is, given a partition~$\mathcal{P}$ we find a small vertex set~$A \subseteq V(G)$ which is balanced (i.e. it intersects each class of~$\mathcal{P}$ in the same number of vertices) with the property that for every small balanced set~$S \subseteq V(G)$,~$A \cup S$ can be partitioned into independent transversals.
Thus the problem of finding a partition into independent transversals is reduced to the problem of finding a collection of disjoint independent transversals covering almost all vertices, which we find in Section~\ref{section:partialstrongcolourings}.
Then the pieces of the proof are put together in Section~\ref{section:mainproof}.

Throughout the proof, we will use the following notation.
Given~$a$,~$b$,~$c$ reals with $c > 0$,~$a = b \pm c$ means that~$b - c \leq a \leq b + c$.
We write~$x \ll y$ to mean that for all $y \in (0, 1]$ there exists $x_0 \in (0,1)$ such that for all $x \leq x_0$ the following statements hold.
Hierarchies with more constants are defined in a similar way and are to be read from the right to the left.

\section{Absorption} \label{section:absorption}

The aim of this section is to prove Lemma~\ref{lemma:absorbing}, the existence of an absorbing set.
First we need the following simple lemma.

\begin{lemma} \label{lemma:twotransversals}
	Let~$G$ be a graph and let~$\mathcal{P} = \{V_1, \dots, V_k\}$ be a partition of~$V(G)$ such that $|V_i| > 2 \Delta(G)$ for all $i \in [k-1]$.
	Then for any $v_k,v_k' \in V_k$, there exists an independent transversal~$T$ of~$\{V_1, \dots, V_{k-1}\}$ such that~$T \cup \{v_k\}$ and~$T \cup \{ v'_k \}$ are independent transversals of~$\mathcal{P}$. 
\end{lemma}

To prove Lemma~\ref{lemma:twotransversals} we will use the following result of Haxell~\cite{Haxell2001}, which was stated differently~\cite{Haxell1995} (see the discussion after Corollary 15 in~\cite{Haxell2016}).

\begin{lemma}[Haxell~{\cite[Theorem 3]{Haxell2001}}] \label{lemma:haxellcondition}
	Let~$G$ be a graph and let~$\mathcal{P} = \{ V_1, \dotsc, V_k \}$ be a partition of~$V(G)$.
	If each~$I \subseteq [k]$ satisfies $\left| \bigcup_{i \in I} V_i \right| > (2|I| - 2) \Delta(G)$, then there exists an independent transversal of~$\mathcal{P}$.
\end{lemma}

Now we prove Lemma~\ref{lemma:twotransversals}.

\begin{proof}[Proof of Lemma~\ref{lemma:twotransversals}]
	For every $i \in [k-1]$, let $V'_i \dfn V_i \setminus ( N(v_k) \cup N(v'_k ))$.	
	Let $\mathcal{P}' \dfn \{ V'_1, \dotsc, V'_{k-1} \}$ and $G' \dfn G[ \bigcup_{i \in [k-1]} V'_i ]$.
	Clearly it is enough to find an independent transversal of~$\mathcal{P}'$ in~$G'$.	
	For every non-empty $I \subseteq [k-1]$, we have that \begin{align*}
	\left| \bigcup_{i \in I} V'_i \right|
	& \ge (2 \Delta(G)  + 1) |I| - |N(v_k) \cup N(v'_k)| \\
	& \ge (2 \Delta(G) + 1) |I| - 2 \Delta(G) \\
	& > (2 |I| - 2) \Delta(G) \ge (2|I| - 2) \Delta(G').
	\end{align*} By Lemma \ref{lemma:haxellcondition},~$G'$ has an independent transversal of~$\mathcal{P}'$, as desired.
\end{proof}

By supersaturation, Lemma~\ref{lemma:twotransversals} implies the following corollary.
We omit its proof since it is quite standard (see, e.g.,~\cite[Section 2]{MubayiZhao2007}).

\begin{corollary} \label{cor:supersaturation}
	For all integers~$k \ge 3$ and~$\eps > 0$, there exists $\eta = \eta(k, \eps) > 0$ and $r_0 = r_0(k,\eps)$ such that the following holds for all~$r \ge r_0$:
	let~$G$ be a graph and let~$\mathcal{P} = \{ V_1, \dotsc, V_k \}$ be a partition of~$V(G)$ with classes of size $r \ge (2 + \eps) \Delta(G)$.
	Then for any two vertices $v_k, v'_k \in V_k$, there exist at least~$\eta r^{k-1}$ independent transversals~$T$ of~$\{V_1,\dotsc,V_{k-1}\}$ such that~$T \cup \{ v_k \}$ and~$T \cup \{ v'_k \}$ are independent transversals of~$\mathcal{P}$.
\end{corollary}

% select m large enough and select a subset of size m randomly inside each V_i, since wlog every vertex has neighbours only in the other clusters then using Chernoff we get that for every vertex the number of choices such that D > 2m+1 is less than (r choose m-1) (r choose m)^(k-1) exp( - Omega(m), summing over all vertices then everything is less than (r choose m)^(k) exp( - Omega(m)) and choosing m large (not depending on r) we get exp( -Omega(m)) < 1/2. for each choice that works we have common transversal for two given vertices, and correcting for the overcount we get the result

To continue, we need to recall the following versions of the Chernoff inequalities.

\begin{lemma}[Chernoff's inequalities (see, e.g., {\cite[Theorem 2.8]{JansonLuczakRucinski2000}})] \label{lemma:chernoff}
	Let~$X$ be a generalised binomial random variable, that is,~$X$ is the sum of independent Bernoulli random variables, possibly with different parameters.
	For every $0 < \lambda \leq \expectation[X]$, \begin{equation}
	\prob[ |X - \expectation[X]| > \lambda ] \leq 2 \exp \left( - \frac{\lambda^2}{4\expectation[X]} \right). \label{eq:chernoffsimple}
	\end{equation}
	Also, for every~$\lambda > 0$,
	\begin{equation} \prob[ X - \expectation[X] > \lambda ] \leq \exp \left( - \frac{\lambda^2}{2 (\expectation[X] + \lambda / 3)} \right). \label{eq:chernoffsumofhypergeometrics} \end{equation}
\end{lemma}

Since every hypergeometric distribution is a sum of independent Bernoulli variables (see, e.g.,~\cite[Theorem 2.10]{JansonLuczakRucinski2000}), the inequalities \eqref{eq:chernoffsimple} and \eqref{eq:chernoffsumofhypergeometrics} also hold when~$X$ is a sum of independent hypergeometric variables.

We can now prove our absorbing lemma.
Given a graph~$G$ and a partition~$\mathcal{P}=\{ V_1,\dotsc,V_k\}$ of~$V(G)$, a subset~$S \subseteq V(G)$ is \emph{$\mathcal{P}$-balanced} (or just \emph{balanced}, if~$\mathcal{P}$ is clear from the context) if $|S \cap V_i| = |S \cap V_j|$ for every $i, j \in [k]$.

\begin{lemma}[Absorbing lemma] \label{lemma:absorbing}
	For all integers~$k \ge 3$ and~$\eps > 0$, there exist $0 < \beta \ll \alpha \ll \eps$ and $r_0 = r_0(\eps, k)$ such that the following holds for all~$r \ge r_0$.
	Let~$G$ be a graph and let~$\mathcal{P}=\{ V_1,\dotsc,V_k\}$ be a partition of~$V(G)$ with classes of size $r \ge (2 + \eps) \Delta(G)$.
	Then there exists a $\mathcal{P}$-balanced set~$A \subseteq V(G)$ of size at most~$\alpha n$ such that for every~$\mathcal{P}$-balanced set~$S \subseteq V(G)$ of size at most~$\beta n$,~$A \cup S$ can be partitioned into independent transversals of~$\mathcal{P}$.	
\end{lemma}

\begin{proof}
	Let $1/r_0 \ll \gamma \ll \eta \ll 1/k, \eps$.
	Let~$m \dfn k^2$.
	Given a balanced $k$-subset~$S$ of~$V(G)$, an \emph{absorbing set~$A$ for~$S$} is a $m$-subset of~$V(G)$, disjoint from~$S$, such that both~$G[A]$ and~$G[A \cup S]$ can be partitioned into independent transversals of~$\mathcal{P}$.	
	For any balanced $k$-set~$S$, let~$\mathcal{L}(S)$ be the family of absorbing sets for~$S$.
	
	\begin{claim}
		For each balanced $k$-subset~$S$ of~$V(G)$, $|\mathcal{L}(S)| \ge \gamma \binom{r}{k}^k$.
	\end{claim}
	
	\begin{proofclaim}
		Let~$S = \{ s_1, \dotsc, s_k \}$ with $s_i \in V_i$ for every $i \in [k]$.
		
		A tuple~$(T, U_1, \dotsc, U_k)$ is \emph{good} if $S, T, U_1, \dotsc, U_k$ are pairwise disjoint, $T = \{ t_1, \dotsc, t_k \}$ is an independent transversal of~$\mathcal{P}$ and, for every $i \in [k]$, $t_i \in V_i$ and both $U_i \cup \{ s_i \}$ and $U_i \cup \{ t_i \}$ are independent transversals of~$\mathcal{P}$.
		Clearly, if~$(T, U_1, \dotsc, U_k)$ is good, then~$A = T \cup \bigcup_{i \in [k]} U_i$ is an absorbing set for~$S$.
		
		Let~$t_1$ be an arbitrary vertex of~$V_1 \setminus \{ s_1 \}$, for which we have at least $r-1 \ge r/2$ different possible choices.
		By Corollary~\ref{cor:supersaturation}, there exist~$\eta r^{k-1}$ independent transversals~$T'$ of~$\{ V_2, \dotsc, V_k \}$ such that both~$T' \cup \{ s_1 \}$ and~$T' \cup \{ t_1 \}$ are independent transversals of~$\mathcal{P}$.
		By ignoring those~$T'$ which have non-empty intersection with~$S$, we have at least $\eta r^{k-1}/2$ different possible choices for~$T'$.
		Set~$T = \{ t_1 \} \cup T'$.
		Repeating the same argument with~$s_i, t_i$ we can find~$\eta r^{k-1}/2$ choices for~$U_i$, for every $i \in [k]$.
		Therefore, there are at least~$(\eta^{k+1} / 2^{k+2}) r^m$ good tuples.
		Using~$\gamma \ll \eta$, we find those good tuples yield at least~$\gamma \binom{r}{k}^k$ different absorbing sets for~$S$, as desired.
	\end{proofclaim}
	
	Recall that $m = k^2$ and choose a family~$\mathcal{F}$ of balanced $m$-sets by including each one of the~$\binom{r}{k}^{k}$ balanced $m$-sets independently at random with probability \[ p \dfn \frac{\gamma r}{16 k^2 \binom{r}{k}^k}. \]
	By Chernoff's inequality~\eqref{eq:chernoffsimple}, with probability $1 - o(1)$ we have that \begin{equation} |\mathcal{F}| \leq \frac{\gamma r}{8 k^2}, \label{eq:absorbing1} \end{equation} and for every balanced $k$-set~$S$, \begin{equation} |\mathcal{L}(S) \cap \mathcal{F} | \ge \frac{\gamma^2 r}{32 k^2}. \label{eq:absorbing2} \end{equation}
	
	We say a pair~$(A_1, A_2)$ of $m$-sets is \emph{intersecting} if $A_1 \neq A_2$ and $A_1 \cap A_2 \neq \emptyset$.
	We say~$\mathcal{F}$ \emph{contains} a pair~$(A_1, A_2)$ if $A_1, A_2 \in \mathcal{F}$.
	The expected number of intersecting pairs contained in~$\mathcal{F}$ is at most \[ \binom{r}{k}^k k^2 \binom{r}{k-1} \binom{r}{k}^{k-1} p^2 \leq \frac{\gamma^2 r}{2^7 k^2}. \]
	By Markov's inequality, with probability at least~$1/2$ the number of intersecting pairs contained in~$\mathcal{F}$ is at most~$\gamma^2 r / (2^6 k^2)$.
	Therefore, with positive probability~$\mathcal{F}$ satisfies \eqref{eq:absorbing1} and \eqref{eq:absorbing2} and contains at most~$\gamma^2 r / (2^6 k^2)$ intersecting pairs.
	
	By removing one $m$-set of every intersecting pair in~$\mathcal{F}$, we obtain a family~$\mathcal{F}'$ of pairwise disjoint balanced $m$-sets such that for every balanced $k$-set~$S$, \[ |\mathcal{L}(S) \cap \mathcal{F}'| \ge \frac{\gamma^2 r}{32 k^2} - \frac{\gamma^2 r}{2^6 k^2} \ge \frac{\gamma^2 r}{64 k^2}. \]
	
	Let $A \dfn \bigcup_{F \in \mathcal{F}'} F$.
	Define $\alpha \dfn \gamma / (8 k^2)$ and $\beta \dfn \gamma^2 / 64 k^2$.
	Note that~$A$ has size at most $ k |\mathcal{F}'| \leq k |\mathcal{F}| \leq \alpha n$.
	For every balanced $S \subseteq V(G)$ of size at most~$\beta n$, we can partition it into at most $\beta r \leq \gamma^2 r / (64 k^2)$ balanced $k$-sets, so it is possible to assign a different absorbing $m$-set in~$\mathcal{F}'$ to each one of these sets.
	Hence,~$G[A \cup S]$ can be partitioned into independent transversals of~$\mathcal{P}$, as desired.
\end{proof}

\section{Partial strong colourings} \label{section:partialstrongcolourings}

Let~$G$ be a graph and $\mathcal{P} = \{ V_1, \dotsc, V_k \}$ a partition of~$V(G)$ with classes of size~$r$.
A \emph{$t$-partial strong colouring} of~$G$ with respect to~$\mathcal{P}$ is a collection of~$t$ disjoint independent transversals of~$\mathcal{P}$.
If $\schrom(G) = r$, then~$G$ has a $r$-partial strong colouring with respect to~$\mathcal{P}$.
The aim of this section is to show the existence of $(1 - \delta)r$-partial strong colourings of~$\mathcal{P}$.

\begin{lemma} \label{lemma:almostperfectcolouring}
	For all integers $k \ge 3$ and $\delta, \eps > 0$, there exists $r_0 = r_0(k, \delta, \eps)$ such that the following holds for all $r \ge r_0$:
	Let~$G$ be a graph and~$\mathcal{P}$ be a partition of~$V(G)$ with~$k$ classes of size $r \ge (2 + \eps) \Delta(G)$.
	Then there exists a $(1 - \delta) r$-partial strong colouring of~$G$ with respect to~$\mathcal{P}$. 
\end{lemma}

We need two extra ingredients to prove Lemma~\ref{lemma:almostperfectcolouring}.
The first ingredient will be a result in fractional matchings that will follow from Theorem~\ref{theorem:fractionalcolouring}.
Recall that a fractional colouring of a graph~$G$ is a function~$w$ that assigns weights in~$[0,1]$ to the independent sets of~$G$, with the condition that for every vertex $v \in V(G)$, $\sum_{I \ni v} w(I) \ge 1$.
The fractional chromatic number of~$G$ is the minimum of $\sum w(I)$ over all fractional colourings of~$G$, where the sum ranges over all independent sets of~$G$.
Note that if a graph~$G$ is fractionally strongly $r$-colourable, then for every partition~$\mathcal{P}$ of~$V(G)$ with classes of size~$r$, every optimal fractional colouring~$w$ for~$\mathcal{P}$ is supported precisely on independent transversals of~$\mathcal{P}$ and for every vertex $v \in V(G)$, $\sum_{I \ni v} w(I) = 1$.
Thus we have the following corollary of Theorem~\ref{theorem:fractionalcolouring}.

\begin{corollary} \label{corollary:fractionalmatching}
	Let~$G$ be a graph and~$\mathcal{P}$ a partition of~$V(G)$ with classes of size $r \ge 2 \Delta(G)$.
	Let~$\mathcal{T}$ be the set of all independent transversals of~$\mathcal{P}$.
	Then there exists $w: \mathcal{T} \rightarrow [0,1]$ such that $\sum_{T \ni v, T \in \mathcal{T}} w(T) = 1$ for all $v \in V(G)$.
\end{corollary}

%The second ingredient is the following theorem of Frankl and Rödl, which guarantees the existence of large matchings in uniform hypergraphs satisfying certain regularity conditions.
%
%\begin{theorem}[Frankl and Rödl~\cite{FranklRodl1985}] \label{theorem:nibble}
%	For all integers $k \ge 2$ and $\delta \ge 0$, $a > 3$, there exists $\tau = \tau(k, \delta)$ and $n_0 = n_0(k, \delta)$ such that the following is true for all $n \ge n_0$: if~$H$ is a $k$-uniform hypergraph of order~$n$ satisfying
%	\begin{enumerate}
%		\item $\deg(v) = (1 \pm \tau ) D$ for all $v \in V(H)$, and
%		\item $\Delta_2(H) := \max_{T \in \binom{V(H)}{2}} \deg(T) < D / (\log n)^a$,
%	\end{enumerate} then~$H$ contains a matching~$M$ covering all but at most $\delta n$ vertices.
%\end{theorem}

The second ingredient is a result that guarantees the existence of large matchings in uniform hypergraphs satisfying certain regularity conditions.
We use the following result of Pippenger~\cite{Pippenger} (see \cite[Theorem 1.1]{KahnKayll1997}).

\begin{theorem}[Pippenger \cite{Pippenger}] \label{theorem:nibble}
	For all integers $k \ge 2$ and $\delta \ge 0$, there exists $D_0 = D_0(k, \delta)$ and $\tau = \tau(k, \delta)$ such that the following is true for all $D \ge D_0$: if~$H$ is a $k$-uniform hypergraph on~$n$ vertices which satisfies
	\begin{enumerate}
		\item $\deg(v) = (1 \pm \tau ) D$ for all $v \in V(H)$, and
		\item $\deg(u,v) < \tau D$, for all distinct $u, v \in V(H)$, %\Delta_2(H) := \max_{T \in \binom{V(H)}{2}}
	\end{enumerate} then~$H$ contains a matching~$M$ covering all but at most $\delta n$ vertices.
\end{theorem}

We now prove Lemma~\ref{lemma:almostperfectcolouring}, whose proof is based on~\cite[Lemma 3.5]{LoMarkstrom2013}.

\begin{proof}[Proof of Lemma~\ref{lemma:almostperfectcolouring}]
	
	Fix $k \ge 3$ and $\delta, \eps > 0$.
	Without loss of generality, suppose $\eps \leq 1$.
	Choose~$r_0$ such that $1/r_0 \ll 1/k, \delta, \eps$.
	Now consider any $r \ge r_0$ and a graph~$G$ on $n = rk$ vertices with $r \ge (2 + \eps) \Delta(G)$.
	Fix a partition~$\mathcal{P}$ of~$V(G)$ with classes of size~$r$.
	Note that $k = |\mathcal{P}| \ge 3$.
	
	The idea is to define a $k$-uniform hypergraph~$H$ on the vertex set~$V(G)$ such that every edge corresponds to an independent transversal of~$\mathcal{P}$ and~$H$ also satisfies the conditions of Theorem~\ref{theorem:nibble}.
	A matching in~$H$ covering all but at most $\delta n$ vertices corresponds precisely to a $(1 - \delta) r$-partial strong colouring of~$\mathcal{P}$.
	
	Fix $\eta_1, \eta_2, \eta_3 \in (0,1)$ such that \begin{align*}
	\eta_1 + \eta_2 & > \eta_3, & 2 \eta_3 &> \eta_1 + \eta_2, \\
	2 \eta_1 + \eta_2 & < 1, & 1 & \leq 3 \eta_2 + 6 \eta_1.
	\end{align*} (For concreteness, $(\eta_1, \eta_2, \eta_3) = (0.1, 0.1, 0.175)$ works).
	Let $m \dfn r^{\eta_1}$.
	
	\begin{claim} \label{claim:tediouscalculations}		
		There exist $r^{1 + \eta_2}$ vertex sets~$R(1)$,$\dotsc$,$R(r^{1 + \eta_2})$ such that
		\begin{enumerate}
			\item \label{item:sizes} for every $j \in [r^{1 + \eta_2}]$,~$R(j)$ is a balanced $mk$-set,
			\item \label{item:singletons} every $v \in V(G)$ is in $r^{\eta_1 + \eta_2} \pm r^{\eta_3}$ many~$R(j)$,
			\item \label{item:pairs} every $\mathcal{P}$-legal $2$-set is in at most two~$R(j)$,
			\item \label{item:triples} every $\mathcal{P}$-legal $3$-set is in at most one~$R(j)$,
			\item \label{item:fractionalcondition} for every $j \in [r^{1 + \eta_2}]$, $m \ge 2 \Delta(G[R(j)]) $.
		\end{enumerate}
	\end{claim}

	To prove Claim~\ref{claim:tediouscalculations}, note that properties \ref{item:sizes}--\ref{item:fractionalcondition} hold with high probability if each~$R(j)$ is a random balanced $mk$-set, chosen uniformly and independently.
	See Appendix~\ref{appendix:tediouscalculations} for the precise calculations.

	Let~$R(1)$,$\dotsc$,$R(r^{1 + \eta_2})$ be given by Claim~\ref{claim:tediouscalculations}.
	By \ref{item:fractionalcondition} and Corollary~\ref{corollary:fractionalmatching}, for each $j \in [r^{1 + \eta_2}]$ there exists a function~$w^j$ that assigns weights in~$[0,1]$ to the independent transversals of~$\mathcal{P}$ contained in~$G[R(j)]$, such that for every $v \in V(G[R(j)])$, $\sum_{T \ni v} w^j(T) = 1$.
	Now we construct a random $k$-uniform graph~$H$ on~$V(G)$ such that each independent transversal~$T$ of~$\mathcal{P}$ is randomly independently chosen as an edge of~$H$ with \[ \prob[ T \in H  ] = \begin{cases}
	w^{j_T}(T) & \text{if $T \subseteq G[R(j_T)]$ for some $j_T \in [r^{1 + \eta_2}]$,} \\
	0 & \text{otherwise.}
	\end{cases} \]
	Note that~$j_T$ is unique by~\ref{item:triples} as $k \ge 3$, so~$H$ is well-defined.
	For $v \in V(G)$, let $J_v = \{ j : v \in R(j) \}$ and so $|J_v| = r^{\eta_1 + \eta_2} \pm r^{\eta_3}$ by~\ref{item:singletons}.
	For every $v \in V(G)$, let~$E^j_v$ be the set of independent transversals in~$G[R(j)]$ containing~$v$.
	Thus, for $v \in V(G)$,~$\deg_{H}(v)$ is a generalised binomial random variable with expectation \[ \expectation[ \deg_{H}(v) ] = \sum_{j \in J_v} \sum_{T \in E^j_v} w^j(T) = |J_v| = r^{\eta_1 + \eta_2} \pm r^{\eta_3}. \]	
	Similarly, for every $\mathcal{P}$-legal $2$-set $\{u, v\}$,
	\[ \expectation[ \deg_{H}(u,v) ] = \sum_{j \in J_u \cap J_v} \sum_{T \in E^j_u \cap E^j_v} w^j(T) \leq |J_u \cap J_v| \leq 2 \] by \ref{item:pairs}.
	For every $2$-set~$\{u,v\}$ that is not $\mathcal{P}$-legal, $\deg_H(u,v) = 0$.
	Fix $\eta_6 \in (0,1)$ such that $\eta_1 + \eta_2 > \eta_6 > \eta_3$.
	By using Chernoff's inequality \eqref{eq:chernoffsimple}, we may assume that for every $v \in V(G)$ and every $2$-set $\{u,v\} \subseteq V(G)$, \[ \deg_{H}(v) = r^{\eta_1 + \eta_2} \pm r^{\eta_6}, \qquad \deg_{H}(u,v) < r^{\eta_1}. \]
	Thus~$H$ satisfies the hypothesis of Theorem~\ref{theorem:nibble} and the proof is completed.	
\end{proof}

\section{Proof of Theorem~\ref{theorem:main}} \label{section:mainproof}

\begin{proof}[Proof of Theorem~\ref{theorem:main}]
Let $1 / r_0 \ll \beta \ll \alpha \ll 1/k, \eps$ and consider a graph~$G$ on~$n$ vertices and a partition $\mathcal{P} = \{ V_1, \dotsc, V_k \}$ with classes of size $r \ge (2 + \eps) \Delta(G)$.
Note that $n = rk$.

By Lemma~\ref{lemma:absorbing}, there exists a balanced set~$A$ of size at most $\alpha n$ such that for every balanced set~$S$ of size at most $\beta n$, $G[A \cup S]$ can be partitioned into independent transversals of~$\mathcal{P}$.
Remove~$A$ from~$G$ to obtain a graph~$G''$, together with a partition $\mathcal{P}'' = \{ V''_1, \dotsc, V''_k \}$ obtained from $V''_i = V_i \setminus A$.
Note $\Delta(G'') \leq \Delta(G)$ and $r'' \dfn |V''_i| = (1 - \alpha) r$ and therefore, $r'' \ge (1 - \alpha) (2 + \eps) \Delta(G) \ge (2 + \eps/2) \Delta(G'')$.

By Lemma~\ref{lemma:almostperfectcolouring}, we obtain a $(1 - \beta)r''$-partial strong colouring of~$G''$ with respect to~$\mathcal{P}''$.
This gives a collection~$\mathcal{T}''$ of disjoint independent transversals of~$\mathcal{P}$ that covers every vertex of~$G''$ except for a set~$S$ of size at most $\beta r'' \leq \beta r$.
Then $G[A \cup S]$ can be covered by a collection~$\mathcal{T}$ of disjoint independent transversals of~$\mathcal{P}$.
Therefore, $\mathcal{T} \cup \mathcal{T''}$ is a spanning collection of disjoint independent transversals of~$\mathcal{P}$, as desired.
\end{proof}

\subsection*{Acknowledgements}
We thank Klas Markström for introducing the problem and Penny Haxell for her valuable comments and insightful discussions.

\bibliography{strongchromatic}

\appendix

\section{Proof of Claim~\ref{claim:tediouscalculations}} \label{appendix:tediouscalculations}

\begin{proof}
Fix $\eta_4, \eta_5 \in (0,1)$ such that
\begin{align*}
2 \eta_3 > \eta_1 + \eta_2 + \eta_4, \\
1 \leq 3 \eta_2 + 6 \eta_1 + \eta_5.
\end{align*} (For concreteness, if $(\eta_1, \eta_2, \eta_3) = (0.1, 0.1, 0.175)$, then $(\eta_4, \eta_5) = (0.15, 0.1)$ works).
Recall that $m = r^{\eta_1}$.
Let $p \dfn r^{-1 + \eta_1}$ and note that $m = pr$.
For every $i \in [k]$ and $j \in [r^{1 + \eta_2}]$, choose~$R^j_i$ to be a subset of~$V_i$ of size~$m$, chosen independently and uniformly at random.
Let $R(j) \dfn \bigcup_{i \in [k]} R^j_i$.
Clearly, \ref{item:sizes} holds.
Now we show that each of \ref{item:singletons}--\ref{item:fractionalcondition} holds with high probability, and so the desired $R(1), \dotsc, R(r^{1 + \eta_2})$ exist.

Consider $j \in [r^{1 + \eta_2}]$.
Consider an arbitrary $x \in V(G)$ and let $d \dfn \deg_G(x)$ and $d_i \dfn |N(x) \cap V_i| $ for every $i \in [t]$.
Then $d_1 + \dotsc + d_k = d \leq \Delta(G) \leq r/(2 + \eps)$.
Let $D \dfn \deg_{G[R(j)]}(x)$ and $D_i \dfn |N(x) \cap R^j_i|$ for all $i \in [k]$.
Note that $D = D_1 + \dotsc + D_k$, where the~$D_i$ are independent hypergeometric random variables with parameters $m, d_i, r$ each.
Then $\expectation[D] = pd$, and using Lemma~\ref{lemma:chernoff} \eqref{eq:chernoffsumofhypergeometrics}, we get
\begin{align*}
\prob[ 2 D > m ]
% & = \prob\left[ D > \frac{m}{2} \right] \\
& = \prob\left[ D - \expectation[D] > \frac{m}{2} - pd \right]
%& \leq \prob\left[ D - \expectation[D] > pr \left( \frac{1}{2} - \frac{1}{2 + \eps} \right) \right] \\
\leq \prob\left[ D - \expectation[D] > \frac{pr \eps}{6} \right] \\
& \leq \exp\left( - \frac{p^2 r^2 \eps^2}{6 ( 12 pd + 4 p r \eps ) } \right) 
%& \leq \exp\left( - \frac{p r \eps^2}{36 + 24 \eps } \right) \\
% & = \exp( - \Omega(pr) ) \\
 \leq \exp( - \Omega(r^{\eta_1})).
\end{align*}
Hence, with probability $1 - \exp(- \Omega(r^{\eta_1}))$ we have \[ |R^j_i| = m \ge 2 \Delta(G[R]) \text{ for all } i \in [k], \]
that is, \ref{item:fractionalcondition} holds with high probability.

Now we check \ref{item:singletons}--\ref{item:triples}.
Note that for every $x \in V(G)$ and every $j \in [r^{1 + \eta_2}]$, $\prob[x \in R(j)] = p$.
For a $\mathcal{P}$-legal subset $S \subseteq V(G)$, let \[ Y_S \dfn |\{ j : S \subseteq R(j) \}|. \]
Since the probability that a particular $R_i \subseteq V_i$ intersects~$S$ is~$p$, \[ \prob[ S \subseteq R(j) ] = p^{|S|}, \] and therefore $\expectation[ Y_s ] = r^{1 + \eta_2} p^{|S|} = r^{1 + \eta_2 - (1 - \eta_1) |S|}$.
By Lemma~\ref{lemma:chernoff}, with probability at least $1 - \exp( - \Omega(r^{\eta_4}) )$, $Y_v = r^{\eta_1 + \eta_2} \pm r^{\eta_3}$ for every $v \in V(G)$, implying \ref{item:singletons} holds with high probability.

Let $Z_2 \dfn | \{ S \in \binom{V(G)}{2}, S \text{ is $\mathcal{P}$-legal} : Y_S \ge 3 \} |$, and observe that \begin{equation} \expectation(Z_2) < \binom{k}{2} r^2 (r^{1 + \eta_2})^3 p^6 \leq k^2 r^{-1 + 3 \eta_2 + 6 \eta_1} \leq r^{- \eta_5}. \label{eq:z2} \end{equation}
Let $Z_3 \dfn | \{ S \in \binom{V(G)}{3} , S \text{ is $\mathcal{P}$-legal} : Y_S \ge 2 \} |$, and observe that \begin{equation} \expectation(Z_3) < \binom{k}{3} r^3 (r^{1 + \eta_2})^2 p^6 \leq k^3 r^{-1 + 2 \eta_2 + 6 \eta_1} \leq r^{-\eta_5}. \label{eq:z3} \end{equation}
Together with Markov's inequality, \eqref{eq:z2} and \eqref{eq:z3} imply that \ref{item:pairs} and \ref{item:triples} hold with high probability, respectively.
\end{proof}

\end{document}